\documentclass[12pt]{amsart}

\newtheorem{theorem}{Theorem}[section]
\newtheorem{lemma}[theorem]{Lemma}

\theoremstyle{plain}
\newtheorem{definition}[theorem]{Definition}

\newtheorem{question}[theorem]{Question}

\theoremstyle{definition}

\theoremstyle{remark}

\numberwithin{equation}{section}

\usepackage{fancyhdr}
\usepackage{amscd}
\usepackage{amsmath}
\usepackage{amsthm}
\usepackage{amssymb}

\begin{document}

\title{Eliminate obstructions: curves on a 3-fold}

\author{Sen Yang}
\address{Yau Mathematical Sciences Center, Tsinghua University
\\
Beijing, China}
\email{syang@math.tsinghua.edu.cn; senyangmath@gmail.com}

\date{}

\maketitle

\begin{abstract}
By using K-theory, we reinterpret and generalize an idea on eliminating obstructions to deforming cycles, which is known to Mark Green and Phillip Griffiths \cite{GGtangentspace} and TingFai Ng \cite{Ng}(for the divisor case). 

As an application, we show  how to eliminate obstructions to deforming curves on a 3-fold. This answers affirmatively an open question by TingFai Ng \cite{Ng}.
\end{abstract}

\tableofcontents

\section{Introduction to Ng's question}
\label{Introduction}

Let $X$ be a nonsingular projective variety over a field $k$ of characteristic $0$. For $Y \subset X$ a subvariety of codimension $p$, $Y$ can be considered as an element of the Hilbert scheme $\mathrm{Hilb}(X)$ and the Zariski tangent space $\mathrm{T}_{Y}\mathrm{Hilb}(X)$ can be identified with $H^{0}(N_{Y/X})$, where $N_{Y/X}$ is the normal sheaf.  It is well-known 
that $\mathrm{Hilb}(X)$ may be nonreduced at $Y$, the deformation of $Y$ may be obstructed. 

However, Green-Griffiths predicts that we can eliminate obstructions in their program \cite{GGtangentspace}(page 187-190), by considering $Y$ as a \textbf{cycle}. That is, instead of considering $Y$ as an element of $\mathrm{Hilb}(X)$, considering $Y$ as an element of the cycles class group $Z^{p}(X)$ can eliminate obstructions.
For $p=1$, Green-Griffiths' idea was realized by TingFai Ng in his Ph.D thesis \cite{Ng}:
\begin{theorem}  [Theorem 1.3.3 in \cite{Ng}] \label{theorem: Ng's theorem}
The divisor class group $Z^{1}(X)$ is smooth. 
\end{theorem}
Ng proved that for $Y \in Z^{1}(X)$, we can lift $Y$(as a cycle) to higher order successively. Ng's method is to use the {\it semi-regularity} map,  we sketch it briefly as follows and refer to \cite{Ng} for details.

 Let $X$ be a nonsingular projective variety over a field $k$ of characteristic $0$. For $Y \subset X$ a locally complete intersection of codimension $p$, Bloch constructs the {\it semi-regularity} map in \cite{Bloch}, 
\[
 \pi: H^{1}(Y, N_{Y/X}) \to H^{p+1}(X, \Omega^{p-1}_{X/k}).
\]
In particular, for $p=1$, this map $\pi: H^{1}(Y, N_{Y/X}) \to H^{2}(O_{X})$(due to Kodaira-Spencer \cite{KS}) agrees with the boundary map in the long exact sequence 
\[
 \cdots \to  H^{1}(O_{X}(Y))  \to H^{1}(Y, N_{Y/X}) \to H^{2}(O_{X})  \to \cdots,
\]
associated to the short exact sequence:
\[
 0 \to O_{X} \to O_{X}(Y) \to N_{Y/X} \to 0.
\]

Suppose $Y'$ is a first order infinitesimal deformation of $Y$, which is obstructed to second order, Ng uses the following method to eliminate obstructions.
Let $W$ be an ample divisor such that $H^{1}(O_{X}(Y+W))=0$. Since the subscheme $Y \cup W$ is still a locally complete intersection, we have the {\it semi-regularity} map $\pi: H^{1}(Y \cup W, N_{Y \cup W /X}) \to H^{2}(O_{X})$, which agrees with the boundary map in the long exact sequence 
\[
 \cdots \to  H^{1}(O_{X}(Y+W))  \to H^{1}(Y \cup W, N_{Y \cup W/X}) \xrightarrow{\pi} H^{2}(O_{X})  \to \cdots.
\]

Since $H^{1}(O_{X}(Y+W))=0$, the kernel of $\pi$ is $0$, so $Y \cup W$ is {\it semi-regular} in $X$. According to Kodaira-Spencer \cite{KS}, see Theorem 1.2 of \cite{Bloch}, the Hilbert scheme $\mathrm{Hilb(X)}$ is smooth at the point corresponding to $Y \cup W$. Let $(Y \cup W)^{'}$ be the first order infinitesimal deformation of $Y \cup W$ satisfying $(Y \cup W)^{'}\mid_{Y}=Y'$, then $(Y \cup W)^{'}$ can be deformed to second order.

As an algebraic cycle, $Y$ can be written as a formal sum
\[
  Y = (Y+W) - W \in Z^{1}(X).
\]
To deform $Y$ is equivalent to deforming $(Y+W)$ and $W$ respectively. However, the classical definition of algebraic cycles can't distinguish nilpotent, we can't deform algebraic cycles $(Y+W)$ or $W$ directly. Since $(Y+W)$ and $W$ are algebraic cycles associated to the subschemes $Y \cup W$ and $W$ respectively, we deform the subschemes $Y \cup W$ and $W$ to $(Y \cup W)'$ and $W'$,  and consider them as deformations of  algebraic cycles $(Y+W)$ and $W$.

To avoid bringing new obstructions, we fix $W$, that is, we take $W^{'}=W \subset X \subset X[\varepsilon]/(\varepsilon^2)$ as a first order infinitesimal deformation of $W$, and $W^{'}$ can be deformed to second order  $W^{''}=W \subset X \subset X[\varepsilon]/(\varepsilon^3)$. 

We consider $(Y \cup W)'-W$ as a first order deformation of $Y$, it satisfies 
\[
((Y \cup W)'-W)\mid_{Y}=Y'.
\]
Moreover, $(Y \cup W)'-W$  deforms to second order.

Ng's method suggests that the following interesting idea,

\textbf{Idea}($\ast$): When the deformation of $Y$ is obstructed, find $W$ such that
\begin{itemize}
\item 1. $W$ helps to eliminate obstructions,  \\
\item 2. $W$ doesn't bring new obstructions.
\end{itemize}

\textbf{Remark:} The readers can see the above argument is not very satisfactory, since we use $(Y \cup W)'$ as the deformation of  the algebraic cycle $Y+W$.
It is better if we can deform algebraic cycles directly, which is one of the topics in the program by Green-Griffiths \cite{GGtangentspace}. In \cite{Y-2,Y-3}, we propose Milnor K-theoretic cycles which can detect nilpotent and use them to deform cycles. We will use Milnor K-theoretic cycles to reinterpret and extend Ng's idea in this note.

In \cite{Ng}, TingFai Ng asks whether we can extend the above idea beyond divisor case, e.g. , curves on a 3-fold.
Suppose $C^{'} \subset X[\varepsilon]/(\varepsilon^2)$ is a first order infinitesimal deformation of a curve $C$ on a 3-fold $X$, while the infinitesimal deformation of $C^{'}$ to $X[\varepsilon]/(\varepsilon^{3})$ may be obstructed(as a subscheme), Ng asks whether a deformation of $C^{'}$ as a \textbf{cycle}\footnote{The phrase ``a deformation of $C^{'}$ as a cycle" is unclear, we will give a definition in Definition \ref{definition: deformation}} always exists:
\begin{question} [ Section 1.5 of  \cite{Ng}] \label{question: OrigNg}
Given a smooth closed curve $C$ in a 3-fold $X$ and a normal vector field $v$, we wish  to know whether it is always possible to find a nodal curve $\tilde{C}$ in X, of which $C$ is a component, (i.e. $\tilde{C} = C \cup D$ for some residue curve $D$) and a normal vector field $\tilde{v}$ on $\tilde{C}$ such that

(1).  $\tilde{v}|_{C}= v$ ,

(2). the first order deformation given by ($\tilde{C}$, $\tilde{v}$) extends to second order,  and 

(3). the first order deformation given by ($D$, $v'$)extends to second order.

Cycle-theoretically, we have  $(C, v)$ = ($\tilde{C}$, $\tilde{v}$)-($D$, $v'$). So we are asking whether $(C, v)$ as a first order deformation of cycles always extends to second order.

\end{question}

Ng doesn't specify what $v'$ should be, one might guess $v'= \tilde{v}|_{D}$. It turns out that this is not the case later, in fact, we should take $v'=0$(meaning we fix $D$, so it doesn't bring new obstructions). Moreover, if $(C, v)$ is obstructed, we can see that $(C, v)$ $\neq$ ($\tilde{C}$, $\tilde{v}$)-($D$, $v'$), otherwise, because of $(2)$ and $(3)$ in Question \ref{question: OrigNg}, there is no obstructions to extending $(C, v)$ to second order.  What we can expect is to use ($\tilde{C}$, $\tilde{v}$)-($D$, $v'$), which is another first order deformation of $C$,  to replace $(C, v)$ and extend ($\tilde{C}$, $\tilde{v}$)-($D$, $v'$) to second order. So we modify Ng's Question \ref{question: OrigNg} as follows:
\begin{question} [ Section 1.5 of  \cite{Ng}] \label{question: Ng}
Given a smooth \footnote{We will remove this hypothesis later in Question ~\ref{question: NgRewrite}.} closed curve $C$ in a 3-fold $X$ and a normal vector field $v$, we wish  to know whether it is always possible to find a nodal curve $\tilde{C}$ in X, of which $C$ is a component, (i.e. $\tilde{C} = C \cup D$ for some residue curve $D$) and a normal vector field $\tilde{v}$ on $\tilde{C}$ such that

(1).  $\tilde{v}|_{C}= v$ ,

(2). the first order deformation given by ($\tilde{C}$, $\tilde{v}$) extends to second order,  and 

(3). the first order deformation given by ($D$, $v'$)extends to second order.

\textbf{Cycle-theoretically}, ($\tilde{C}$, $\tilde{v}$)-($D$, $v'$) is a first order deformation of $C$ and we can extend it to second order.

\end{question}

The key to answer Ng's question is to interpret it, especially the word \textbf{Cycle-theoretically}, in an appropriate way. For this purpose, we reformulate Ng's question in the framework of \cite{Y-3, Y-4} and answer it affirmatively in Theorem \ref{theorem: answerToNg}. 

For $X$ a $d$-dimensional smooth projective variety and $Y \subset X$ a subvariety of codimension $p$($1\leqslant p \leqslant d$), considering $Y$ as an element of the cycle class group $Z^{p}(X)$, 
Mark Green and Phillip Griffiths \cite{GGtangentspace}(page 187-190) conjecture that we can deform $Y$(as a cycle) to higher order successively. This has been reformulated and has been answered affirmatively in \cite{Y-3}(Section 3). 

We remark that Ng's Question \ref{question: Ng}
above and its reformulation Question \ref{question: NgRewrite} below, are slightly different from Green-Griffiths' question on obstruction issues reformulated in \cite{Y-3}. That's mainly because we don't know  whether the map $\mu$ in Definition \ref{definition: map1} is surjective or not\footnote{The author learned this subtlety from discussion \cite{Bloch1} with Spencer Bloch}.

\vspace{5mm}

\textbf{Set-up} \  Throughout this note, we consider the following set-up:

$X$ is a nonsingular projective 3-fold over a field $k$ of characteristic $0$, let $Y \subset X$ be a curve with generic point $y$. 
For a point $x \in Y \subset X$, the local ring $O_{X,x}$ is a regular local ring of dimension 3 and the maximal ideal $m_{X,x}$ is generated by a regular sequence $f ,g, h$. We assume $Y$ is generically defined by $(f,g)$,
so the local ring $O_{X,y}=(O_{X,x})_{(f,g)}$.

\textbf{Notations:}

(1). K-theory used in this note will be Thomason-Trobaugh non-connective K-theory, if not stated otherwise. 

(2). For any abelian group $M$, $M_{\mathbb{Q}}$ denotes the image of $M$ in $M \otimes_{\mathbb{Z}} \mathbb{Q}$. 

(3). $(a,b)^{\mathrm{T}}$ denotes the transpose of $(a,b)$.
\section{Reformulate Ng's question}
\label{Ng's question and reformulation}
For $X$ is a nonsingular projective 3-fold over a field $k$ of characteristic $0$,
for each non-negative integre $j$, let $X_{j}$ denote the $j$-th infinitesimally trivial deformation of $X$, i.e., $X_{j}= X \times_{k} \mathrm{Spec}(k[\varepsilon]/ \varepsilon^{j+1})$. In particular, $X_{0}=X$, $X_{1}=X[\varepsilon]/ (\varepsilon^{2})$, and $X_{2}=X[\varepsilon]/ (\varepsilon^{3})$.

Let $Y^{'} \subset X_{1}$ be a first order infinitesimal deformation of $Y$, that is, $Y^{'}$ is flat over $\mathrm{Spec}(k[\varepsilon]/(\varepsilon^{2}))$ and $Y^{'} \otimes_{\mathrm{Spec}(k[\varepsilon]/(\varepsilon^{2}))} \mathrm{Spec}(k) \cong Y$.  $Y'$ is generically given by $(f+\varepsilon f_{1}, g+\varepsilon g_{1})$, where $f_{1}$, $g_{1} \in O_{X,y}$, see \cite{Y-4} for related discussions if necessary. For simplicity, we assume $g_{1}=0$ in the following.

We use $F_{\bullet}(f+\varepsilon f_{1}, g)$ to denote the Koszul complex associated to the regular sequence $f+\varepsilon f_{1}, g$,  which is a resolution of $O_{X_{1}, y}/(f+\varepsilon f_{1}, g)$:
\[
0 \to O_{X_{1},y} \xrightarrow{(g, -f-\varepsilon f_{1})^{\mathrm{T}}} O_{X_{1},y}^{\oplus 2} \xrightarrow{(f+\varepsilon f_{1}, g)} O_{X_{1},y}.
\]

Recall that Milnor K-group with support is rationally defined as certain eigenspaces of K-groups in \cite{Y-2},
\begin{definition}  [Definition 3.2 in \cite{Y-2}] \label{definition:Milnor K-theory with support}
Let $X$ be a finite equi-dimensional noetherian scheme and $x \in X^{(p)}$. For $m \in \mathbb{Z}$, Milnor K-group with support $K_{m}^{M}(O_{X,x} \ \mathrm{on} \ x)$ is rationally defined to be 
\[
  K_{m}^{M}(O_{X,x} \ \mathrm{on} \ x) := K_{m}^{(m+p)}(O_{X,x} \ \mathrm{on} \ x)_{\mathbb{Q}},
\] 
where $K_{m}^{(m+p)}$ is the eigenspace of $\psi^{k}=k^{m+p}$ and $\psi^{k}$ is the Adams operations.
\end{definition}

In our setting, $X$ is a nonsingular projective three-fold over a field $k$ of characteristic 0, $y \in X^{(2)}$,  for each non-negative integer $j$,
\[
 K^{M}_{0}(O_{X_{j},y} \ \mathrm{on} \ y): = K^{(2)}_{0}(O_{X_{j},y} \ \mathrm{on} \ y)_{\mathbb{Q}} \subseteq K_{0}(O_{X_{j},y} \ \mathrm{on} \ y)_{\mathbb{Q}}.
\]

\begin{lemma} \label{lemma: OmitMilnor}
In the notation above,
\[
K^{M}_{0}(O_{X_{j},y} \ \mathrm{on} \ y) = K_{0}(O_{X_{j},y} \ \mathrm{on} \ y)_{\mathbb{Q}}.
\]
\end{lemma}

\begin{proof}
For $j=0$, according to Riemann-Roch without denominator \cite{Soule}, 
\[
K^{(i)}_{0}(O_{X,y} \ \mathrm{on} \ y)_{\mathbb{Q}} \cong K^{(i-2)}_{0}(k(y))_{\mathbb{Q}},
\]
where $k(y)$ is the residue field. This forces $K^{(i)}_{0}(O_{X,y}\ \mathrm{on} \ y)_{\mathbb{Q}}=0$, except for $i=2$. So we have
\[
K^{M}_{0}(O_{X,y} \ \mathrm{on} \ y) = K_{0}(O_{X,y} \ \mathrm{on} \ y)_{\mathbb{Q}}.
\]

For each positive integer $j$, let $K^{(i)}_{0}(O_{X_{j},y} \ \mathrm{on} \ y, \varepsilon)_{\mathbb{Q}}$ denote the relative K-group, that is , the kernel of the natural projection
\[
  K^{(i)}_{0}(O_{X_{j},y} \ \mathrm{on} \ y)_{\mathbb{Q}}  \xrightarrow{\varepsilon=0} 
  K^{(i)}_{0}(O_{X,y} \ \mathrm{on} \ y)_{\mathbb{Q}}.
  \] 
We have proved that, see Corollary 3.11 in \cite{Y-2} or Corollary 9.5 in \cite{DHY}, the relative K-group is isomorphic to local cohomology:
\[
K^{(i)}_{0}(O_{X_{j},y} \ \mathrm{on} \ y, \varepsilon)_{\mathbb{Q}} \cong H_{y}^{2}(\Omega^{\bullet,(i)}_{X/\mathbb{Q}}),
 \]
where
\begin{equation}
\begin{cases}
 \Omega_{X/ \mathbb{Q}}^{\bullet,(i)} & = (\Omega^{{2i-3}}_{X/ \mathbb{Q}})^{\oplus j}, \mathrm{for} \   1  < \ i \leq 2;\\
  \Omega_{X/ \mathbb{Q}}^{\bullet,(i)} & = 0, \mathrm{else}.
\end{cases}
\end{equation} 
This says that $K^{(i)}_{0}(O_{X_{j},y} \ \mathrm{on} \ y, \varepsilon)_{\mathbb{Q}}=0$, except for $i=2$.

Since $K_{0}(O_{X_{j},y} \ \mathrm{on} \ y)_{\mathbb{Q}}=K_{0}(O_{X,y} \ \mathrm{on} \ y)_{\mathbb{Q}} \oplus  K_{0}(O_{X_{j},y} \ \mathrm{on} \ y, \varepsilon)_{\mathbb{Q}}$, one sees that $K^{(i)}_{0}(O_{X_{j},y} \ \mathrm{on} \ y)_{\mathbb{Q}}=0$, except for $i=2$. That is, 
\[
 K^{M}_{0}(O_{X_{j},y} \ \mathrm{on} \ y) = K_{0}(O_{X_{j},y}  \ \mathrm{on} \ y)_{\mathbb{Q}}.
\]
\end{proof}

We identify  the Zariski tangent space $\mathrm{T}_{Y}\mathrm{Hilb}(X)$ 
with $H^{0}(Y, N_{Y/X})$ and recall the following,
\begin{definition}  [Definition 2.4 in \cite{Y-4}] \label{definition: map1}
We define a  map $\mu:  \ H^{0}(Y, N_{Y/X}) \to K_{0}(O_{X_{1},y} \ \mathrm{on} \ y)_{\mathbb{Q}}$ as follows: 
\begin{align*}
\mu:   \ H^{0}(Y, \  & N_{Y/X}) \to K_{0}(O_{X_{1},y}  \ \mathrm{on} \ y)_{\mathbb{Q}} \\
& Y' \longrightarrow  F_{\bullet}(f+\varepsilon f_{1}, g),
\end{align*}
where $F_{\bullet}(f+\varepsilon f_{1}, g)$ is the Koszul complex associated 
to $f+\varepsilon f_{1}, g$.
\end{definition}

Now, we recall  Milnor K-theoretic cycles:
\begin{definition}[Definition 3.4 and 3.15 in \cite{Y-2}] \label{definition: Milnor K-theoretic Chow groups}
Let $X$ be a nonsingular projective 3-fold over a field $k$ of characteristic $0$, the second Milnor K-theoretic cycles on $X$ is defined to be
\[
Z^{M}_{2}(D^{\mathrm{Perf}}(X)) := \bigoplus\limits_{y \in X^{(2)}} K_{0}(O_{X,y}  \ \mathrm{on} \ y)_{\mathbb{Q}}.
\]

For each positive integer $j$, $X_{j}$ denote the $j$-th infinitesimally trivial deformation of $X$, the second Milnor K-theoretic cycles on $X_{j}$ are defined to be:
\[
Z^{M}_{2}(D^{\mathrm{Perf}}(X_{j})) := \mathrm{Ker}(d_{1,X_{j}}^{2,-2}), 
\]
where $d_{1,X_{j}}^{2,-2}$ are the differentials in Theorem \ ~\ref{theorem: firstorder}.
\end{definition}

For each positive integer $j$, the natural map $f_{j}: X_{j-1} \to X_{j}$ 
induces the following commutative diagram, see Section 3.1 of  \cite{Y-3}(page 33),
\[
\begin{CD}
\bigoplus\limits_{y \in X^{(2)}}K_{0}(O_{X_{j}, y} \ \mathrm{on} \ y)_{\mathbb{Q}}    @>f_{j}^{*}>>
\bigoplus\limits_{y \in X^{(2)}}K_{0}(O_{X_{j-1},y} \ \mathrm{on} \ y)_{\mathbb{Q}}  \\ 
@Vd_{1,X_{j}}^{2,-2}VV  @Vd_{1,X_{j-1}}^{2,-2}VV \\ 
 \bigoplus\limits_{x \in X^{(3)}} K_{-1}(O_{X_{j},x} \ \mathrm{on} \ x)_{\mathbb{Q}}   @>f_{j}^{*}>>
\bigoplus\limits_{x \in X^{(3)}}K_{-1}(O_{X_{j-1},x} \ \mathrm{on} \ x)_{\mathbb{Q}},
\end{CD}
\]
so it further induces $f^{\ast}_{j}:  Z^{M}_{2}(D^{\mathrm{perf}}(X_{j})) \to  Z^{M}_{2}(D^{\mathrm{perf}}(X_{j-1}))$.

\begin{definition}  [Definition 3.3 \cite{Y-3}] \label{definition: deformation}
For each positive integer $j$, given $\xi_{j-1} \in Z^{M}_{2}(D^{\mathrm{perf}}(X_{j-1}))$,  an element $\xi_{j} \in Z^{M}_{2}(D^{\mathrm{perf}}(X_{j}))$ is called a deformation of $\xi_{j-1}$,  if $f^{\ast}_{j}(\xi_{j}) = \xi_{j-1}$.
\end{definition}

$\xi_{j-1}$ and $\xi_{j}$ can be formally written as finite sums
\[
 \sum_{y}\lambda_{j-1}\cdot \overline{\{y \}} \ \mathrm{and} \   \sum_{y}\lambda_{j}\cdot \overline{\{ y \}},
\] 
where $\lambda_{j}$'s are perfect complexes such that
$\sum\limits_{y}\lambda_{j} \in \mathrm{Ker}(d_{1,X_{j}}^{2,-2}) \subset \bigoplus\limits_{y \in X^{(2)}}K_{0}(O_{X_{j}, y} \ \mathrm{on} \ y)_{\mathbb{Q}}$.

 When we deform from  $\xi_{j-1}$ to $\xi_{j}$, we  deform the \textbf{coefficient} from $\sum\limits_{y}\lambda_{j}$ to $\sum\limits_{y}\lambda_{j+1}$, in other words, we deform the perfect complexes.

Now, we are ready to rewrite Ng's Question \ref{question: Ng} as follows:
\begin{question} \cite{Ng}  \label{question: NgRewrite}
In the notation of Set-up(page 4), for a first order infinitesimal deformation $Y'$ of $Y$, which is generically given by $(f+\varepsilon f_{1}, g)$,
let $\mu(Y') =F_{\bullet}(f+\varepsilon f_{1}, g) \in K_{0}(O_{X_{1},y} \ \mathrm{on} \ y)_{\mathbb{Q}}$ denote the image of $Y'$ under the map $\mu$ in Definition \ ~\ref{definition: map1}.  

Is it always possible to find an element $\gamma = \mu(Y) + \mu(Z) \footnote{To define $\mu(Y)$, we take $f_{1} = g_{1}=0$ in Definition ~\ref{definition: map1}. $\mu(Z)$ can be defined similarly.} \in Z^{M}_{2}(D^{\mathrm{Perf}}(X)) $, for some curve $Z \subset X$  and a first order  deformation $\gamma'$ of $\gamma$, in the sense of  Definition \ref{definition: deformation} such that 

1. $\gamma' = \mu(Y') + \mu(Z') \in Z^{M}_{2}(D^{\mathrm{Perf}}(X_{1}))$, with $Z'$ a first order infinitesimal deformation of $Z$. So it is obvious that $\gamma' \mid_{Y} = \mu(Y')$. 

(2). the first order deformation $\gamma'$ deforms to second order $\gamma'' \in Z^{M}_{2}(D^{\mathrm{Perf}}(X_{2}))$.

(3). Considering $\mu(Z)$ as a first order deformation of itself, we can 
deform it to second order $\mu(Z)$(meaning we fix $Z$).

Instead of considering $\mu(Y')$\footnote{In general, $\mu(Y')$ is not necessary to be a deformation of $\mu(Y)$ in the sense of  Definition \ref{definition: deformation}, see Theorem \ref{theorem: answerToNg}.}, $(\mu(Y') + \mu(Z'))$-$\mu(Z)$ is a first order deformation of $\mu(Y)$ and we can deform it to second order.
\end{question}

\begin{definition} \label{definition: definingCurveZ}
In the same notation of Set-up(page 4), let $z$ be the point defined by the prime ideal $(h,g) \subset O_{X,x}$, then $z \in X^{(2)}$. We define a curve $Z \subset X$ to be 
\[
Z := \overline{\{z \}}.
\]
\end{definition}

\section{Answer Ng's question} 
\label{Curves on 3-fold-Ng's question}
Recall that $X$ is a nonsingular projective 3-fold over a field $k$ of characteristic $0$, for each non-negative integre $j$, let $X_{j}$ denote the $j$-th infinitesimally trivial deformation of $X$.
\begin{theorem} \cite{Y-2}  \label{theorem: firstorder}
For $X$ a nonsingular projective 3-fold over a field $k$ of characteristic $0$, by taking $q=2$ in $\mathrm{Theorem \ 3.14}$ of \cite{Y-2}, for each positive integer $j$, we have the following commutative diagram:
 \footnote{The reason why we can use $K_{-1}(O_{X_{j}, x} \ \mathrm{on} \ x)_{\mathbb{Q}}$ to replace $K^{M}_{-1}(O_{X_{j}, x} \ \mathrm{on} \ x)$ is similar as Lemma \ref{lemma: OmitMilnor}. }\\

\[
  \begin{CD}
     \bigoplus\limits_{y \in X^{(2)}} H_{y}^{2}((\Omega_{X/ \mathbb{Q}}^{1})^{\oplus j}) @<\mathrm{Ch}<< \bigoplus\limits_{y \in X^{(2)}}K_{0}(O_{X_{j},y} \ \mathrm{on} \ y)_{\mathbb{Q}}  \\
      @V(\partial_{1}^{2,-2})^{j}VV @Vd_{1,X_{j}}^{2,-2}VV  \\
     \bigoplus\limits_{x \in X^{(3)}} H_{x}^{3}((\Omega_{X/ \mathbb{Q}}^{1})^{\oplus j}) @<\mathrm{Ch}< \cong < \bigoplus\limits_{x \in X^{(3)}}K_{-1}(O_{X_{j},x} \ \mathrm{on} \ x)_{\mathbb{Q}}. \\
  \end{CD}
\]
\end{theorem}

To fix notations,  for each non-negative integre $j$, $D^{\mathrm{perf}}(X_{j})$ denotes the derived category obtained from the exact category of perfect complex on $X_{j}$ and $\mathcal{L}_{(i)}(X_{j})$ is defined to be
\[
  \mathcal{L}_{(i)}(X_{j}) := \{ E \in D^{\mathrm{perf}}(X_{j}) \mid \mathrm{codim_{Krull}(supph(E))} \geq -i \},
\]
where the closed subset $\mathrm{supph}(E) \subset X$ is the support of the total homology of the perfect complex $E$.
Let $(\mathcal{L}_{(i)}(X_{j})/\mathcal{L}_{(i-1)}(X_{j}))^{\#}$ denote the idempotent completion of the Verdier quotient $\mathcal{L}_{(i)}(X_{j})/\mathcal{L}_{(i-1)}(X_{j})$,
 \begin{theorem} \cite{B-3} \label{theorem: Balmer theorem}
 In the notation above, for each $i \in \mathbb{Z}$, localization induces an equivalence
\[
 (\mathcal{L}_{(i)}(X_{j})/\mathcal{L}_{(i-1)}(X_{j}))^{\#}  \simeq \bigsqcup_{x_{j} \in X_{j}^{(-i)}}D_{{x_{j}}}^{\mathrm{perf}}(X_{j})
\]
between the idempotent completion of the quotient $\mathcal{L}_{(i)}(X_{j})/\mathcal{L}_{(i-1)}(X_{j})$ and the coproduct over $x_{j} \in X_{j}^{(-i)}$ of the derived category of perfect complexes of $ O_{X_{j},x_{j}}$-modules with homology supported on the closed point $x_{j} \in \mathrm{Spec}(O_{X_{j},x_{j}})$. Consequently, localization induces an isomorphism
\[
 K_{0}((\mathcal{L}_{(i)}(X_{j})/\mathcal{L}_{(i-1)}(X_{j}))^{\#})  \simeq \bigoplus_{x_{j} \in X_{j}^{(-i)}}K_{0}(O_{X_{j},x_{j}} \ \mathrm{on} \ x_{j}).
\]
\end{theorem}

In the notation of  Set-up(page 4), we consider $O_{X_{1},x}/(fh+ \varepsilon w_{1}, g)$, where $w_{1}$ is an arbitrary element of $O_{X,x}$. The Koszul resolution of  $O_{X_{1},x}/ (fh+ \varepsilon w_{1}, g)$, denoted $L'$,
\begin{equation}
0 \to O_{X_{1},x} \xrightarrow{(g,-(fh+\varepsilon w_{1}))^{\mathrm{T}}} O_{X_{1},x}^{\oplus 2} \xrightarrow{(fh+ \varepsilon w_{1}, g)} O_{X_{1},x},
\end{equation}
defines an element of $K_{0}((\mathcal{L}_{-2}(X_{1})/\mathcal{L}_{-3}(X_{1}))^{\#})$.

\begin{theorem} \label{theorem: TheoremKernel1}
In the notation above, $L'$ is a Milnor K-theoretic cycle defined in Definition \ref{definition: Milnor K-theoretic Chow groups},
\[
L' \in Z^{M}_{2}(D^{\mathrm{Perf}}(X_{1})).
\]
\end{theorem}

\begin{proof}
Under the isomorphism in Theorem \ref{theorem: Balmer theorem}(take $j=1$ and $i=-2$),
\[
 K_{0}((\mathcal{L}_{(-2)}(X_{1})/\mathcal{L}_{(-3)}(X_{1}))^{\#})  \simeq \bigoplus\limits_{y \in X^{(2)}}K_{0}(O_{X_{1},y} \ \mathrm{on} \ y), 
\]
$L'$ decomposes into the direct sum $L'_{1}$ and $L'_{2}$:
\[
\begin{cases}
 \begin{CD}
L'_{1}: 0 \to (O_{X_{1},x})_{(f,g)} \xrightarrow{(\dfrac{g}{h},-(f+\varepsilon \dfrac{w_{1}}{h}))^{\mathrm{T}}} (O_{X_{1},x})_{(f,g)}^{\oplus 2} \xrightarrow{(f+\varepsilon \dfrac{w_{1}}{h}, \dfrac{g}{h})} (O_{X_{1},x})_{(f,g)}, \\
L'_{2}: 0 \to (O_{X_{1},x})_{(h,g)} \xrightarrow{(\dfrac{g}{f},-(h+\varepsilon \dfrac{w_{1}}{f}))^{\mathrm{T}}} (O_{X_{1},x})_{(h,g)}^{\oplus 2} \xrightarrow{(h+\varepsilon \dfrac{w_{1}}{f}, \dfrac{g}{f})} (O_{X_{1},x})_{(h,g)}.
\end{CD}
\end{cases}
\]
Since $h^{-1}$ exists in $(O_{X_{1},x})_{(f,g)}$ and $f^{-1}$ exists in $(O_{X_{1},x})_{(g,h)}$, the above two complexes are quasi-isomorphic to the following complexes respectively, 
\[
\begin{cases}
 \begin{CD}
L'_{1}: 0 \to (O_{X_{1},x})_{(f,g)} \xrightarrow{(g,-(f+\varepsilon \dfrac{w_{1}}{h}))^{\mathrm{T}}} (O_{X_{1},x})_{(f,g)}^{\oplus 2} \xrightarrow{(f+\varepsilon \dfrac{w_{1}}{h}, g)} (O_{X_{1},x})_{(f,g)}, \\
L'_{2}: 0 \to (O_{X_{1},x})_{(h,g)} \xrightarrow{(g,-(h+\varepsilon \dfrac{w_{1}}{f}))^{\mathrm{T}}} (O_{X_{1},x})_{(h,g)}^{\oplus 2} \xrightarrow{(h+\varepsilon \dfrac{w_{1}}{f}, g)} (O_{X_{1},x})_{(h,g)},
\end{CD}
\end{cases}
\]
where  $L'_{1} \in K_{0}(O_{X_{1},y} \ \mathrm{on} \ y)$ and $L'_{2} \in K_{0}(O_{X_{1},z} \ \mathrm{on} \ z)$, $z$ is defined in Definition \ref{definition: definingCurveZ}.

The Ch map in Theorem \ref{theorem: firstorder}(take j=1),
\begin{equation}
\mathrm{Ch}: \bigoplus\limits_{y \in X^{(2)}}K_{0}(O_{X_{1},y} \ \mathrm{on} \ y) \to \bigoplus\limits_{y \in X^{(2)}}H_{y}^{2}(\Omega_{X/\mathbb{Q}}^{1}),
\end{equation}
may be described by a beautiful construction of B. Ang\'eniol and M. Lejeune-Jalabert \cite{A-LJ}, see also page 5-6 of \cite{Y-3} for a brief summary. In particular,
the image of  $L'_{1}$ under the Ch map is represented by the following diagram, 
\begin{equation}
\begin{cases}
 \begin{CD}
   F_{\bullet}(f,g) @>>> (O_{X,x})_{(f,g)}/(f,g)@>>> 0  \\
   (O_{X,x})_{(f,g)} @> \frac{w_{1}}{h}dg >> \Omega^{1}_{(O_{X,x})_{(f,g)}/\mathbb{Q}},
 \end{CD}
\end{cases}
\end{equation}
where $d=d_{\mathbb{Q}}$ and $F_{\bullet}(f,g) $ is the Koszul complex
\[
(O_{X,x})_{(f,g)} \xrightarrow{(g,-f)^{\mathrm{T}}}  (O_{X,x})_{(f,g)}^{\oplus 2} \xrightarrow{(f,g)} (O_{X,x})_{(f,g)},
\]
To be precise, the above diagram(3.3) gives an element $\alpha$ in \\ $Ext_{(O_{X,x})_{(f,g)}}^{2}((O_{X,x})_{(f,g)}/(f,g), \Omega^{1}_{(O_{X,x})_{(f,g)}/\mathbb{Q}})$. Noting that 
\[
H_{y}^{2}(\Omega^{1}_{ X/\mathbb{Q}})=\varinjlim_{n \to \infty}Ext_{(O_{X,x})_{(f,g)}}^{2}((O_{X,x})_{(f,g)}/(f,g)^{n}, \Omega^{1}_{(O_{X,x})_{(f,g)}/\mathbb{Q}}),
\]
the image $[\alpha]$ of $\alpha$ under the limit is in $H_{y}^{2}(\Omega^{1}_{X /\mathbb{Q}})$ and it is the image of  $L'_{1}$ under the Ch map(3.2). 

Similarly, the image of $L'_{2}$ under the Ch map(3.2) in $H_{z}^{2}(\Omega^{1}_{X /\mathbb{Q}})$ is represented by the following diagram, denoted $\beta$,
\begin{equation}
\begin{cases}
 \begin{CD}
    F_{\bullet}(h,g) @>>> (O_{X,x})_{(h,g)}/(h,g)@>>> 0  \\
   (O_{X,x})_{(h,g)} @> \frac{w_{1}}{f}dg >> \Omega^{1}_{(O_{X,x})_{(h,g)}/\mathbb{Q}},
 \end{CD}
\end{cases}
\end{equation}
where $d=d_{\mathbb{Q}}$ and $F_{\bullet}(h,g)$ is the Koszul complex
\[
(O_{X,x})_{(h,g)} \xrightarrow{(g,-h)^{\mathrm{T}}} (O_{X,x})_{(h,g)}^{\oplus 2} \xrightarrow{(h,g)} (O_{X,x})_{(h,g)}.
\]

$\partial^{2,-2}_{1}$ in Theorem \ref{theorem: firstorder} maps $\alpha$(diagram 3.3) in $H^{3}_{x}(\Omega^{1}_{X/\mathbb{Q}})$ to :
\[
\begin{cases}
 \begin{CD}
   O_{X,x} @>M_{1}>>  O_{X,x}^{\oplus 3} @>M_{2}>> O_{X,x}^{\oplus 3}  @>M_{3}>> O_{X,x} @>>> O_{X,x}/(f, g,h) @>>> 0  \\
   O_{X,x} @> w_{1} dg >> \Omega^{1}_{O_{X,x}/\mathbb{Q}},
 \end{CD}
\end{cases}
\]
where $M_{1},  M_{2}$ and $M_{3}$ are matrices associated to the Koszul resolution of $O_{X,x}/(f, g,h)$:
\[
M_{1}=\begin{pmatrix} f \\ -g \\ h \end{pmatrix}, M_{2}=\begin{pmatrix} 0  & -h & -g \\ -h & 0 & f \\ g &f &0 \end{pmatrix}, M_{3}=(f,g,h).
\]

Similarly, $\partial^{2,-2}_{1}$ in Theorem \ref{theorem: firstorder}
maps $\beta$(diagram 3.4) in $H^{3}_{x}(\Omega^{1}_{X/\mathbb{Q}})$ to :
\[
\begin{cases}
 \begin{CD}
  O_{X,x} @>N_{1}>>  O_{X,x}^{\oplus 3} @>N_{2}>> O_{X,x}^{\oplus 3}  @>N_{3}>> O_{X,x} @>>> O_{X,x}/(h,g,f) @>>> 0  \\
   O_{X,x} @>w_{1} dg >> \Omega^{1}_{O_{X,x}/\mathbb{Q}},
 \end{CD}
\end{cases}
\]
where $N_{1},  N_{2}$ and $N_{3}$ are matrices associated to the Koszul resolution of $O_{X,x}/(h,g,f)$:
\[
N_{1}=\begin{pmatrix} h \\ -g \\ f \end{pmatrix}, N_{2}=\begin{pmatrix} 0  & -f & -g \\ -f & 0 & h \\ g &h &0 \end{pmatrix}, N_{3}=(h,g,f).
\]

Noting the commutative diagram below 
\[\displaystyle
  \begin{CD}
    O_{X,x} @>M_{1}>>  O_{X,x}^{\oplus 3} @>M_{2}>> O_{X,x}^{\oplus 3}  @>M_{3}>> O_{X,x} @>>> O_{X,x}/(f, g,h) @>>> 0   \\
   @V-1VV @VW_{1}VV @VW_{2}VV @V1VV @V \cong VV \\
   O_{X,x} @>N_{1}>>  O_{X,x}^{\oplus 3} @>N_{2}>> O_{X,x}^{\oplus 3}  @>N_{3}>> O_{X,x} @>>> O_{X,x}/(h,g,f) @>>> 0 , 
  \end{CD}
\]
where $W_{1}$ and $W_{2}$ stand for the following matrices:
\[
W_{1}=\begin{pmatrix} 0 & 0 & -1 \\ 0 & -1 & 0 \\ 1 & 0 & 0 \end{pmatrix}, W_{2}=\begin{pmatrix} 0  & 0 & 1 \\ 0 & 1 & 0 \\ 1 &0 &0 \end{pmatrix},
\]
one can see that  $\partial^{2,-2}_{1}(\alpha)$ and $\partial^{2,-2}_{1}(\beta)$ are negative of each other in $Ext_{O_{X,x}}^{3}(O_{X,x}/(f,g,h), \Omega^{1}_{O_{X,x}/\mathbb{Q}})$. Hence, $\partial^{2,-2}_{1}(\alpha + \beta)$ is $0$ in $H^{3}_{x}(\Omega^{1}_{X/\mathbb{Q}})$. Therefore, $d^{2,-2}_{1,X_{1}}(L') = 0$ because of the commutative diagram in Theorem \ref{theorem: firstorder}(take j=1).

\end{proof}

Now, we consider $O_{X_{2},x}/(fh+ \varepsilon w_{1} + \varepsilon^{2} w_{2}, g)$, where $w_{1}$, $w_{2}$ are arbitrary elements of $O_{X,x}$. The Koszul resolution of  $O_{X_{2},x}/ (fh+ \varepsilon w_{1} + \varepsilon^{2} w_{2}, g)$, denoted $L''$,
\begin{equation}
 0 \to O_{X_{2},x} \xrightarrow{(g,-(fh+\varepsilon w_{1} + \varepsilon^{2} w_{2}))^{\mathrm{T}}} O_{X_{2},x}^{\oplus 2} \xrightarrow{(fh+ \varepsilon w_{1} + \varepsilon^{2} w_{2}, g)} O_{X_{2},x},
\end{equation}
defines an element of $K_{0}(\mathcal{L}_{-2}(X_{2})/\mathcal{L}_{-3}(X_{2}))^{\#})$.

Under the isomorphism in Theorem \ref{theorem: Balmer theorem}(for $j=2$ and $i=-2$),
\[
 K_{0}((\mathcal{L}_{(-2)}(X_{2})/\mathcal{L}_{(-3)}(X_{2}))^{\#})  \simeq \bigoplus\limits_{y \in X^{(2)}} K_{0}(O_{X_{2},y} \ \mathrm{on} \ y),
\]
$L''$ decomposes into the direct sum of  $L''_{1}$ and $L''_{2}$:
\[
\begin{cases}
 \begin{CD}
L''_{1}: 0 \to (O_{X_{2},x})_{(f,g)}  \xrightarrow{(g,-(f+\varepsilon \dfrac{w_{1}}{h} + \varepsilon^{2} \dfrac{w_{2}}{h}))^{\mathrm{T}}} (O_{X_{2},x})_{(f,g)}^{\oplus 2} \xrightarrow{(f+\varepsilon \dfrac{w_{1}}{h} + \varepsilon^{2} \dfrac{w_{2}}{h}, g)} (O_{X_{2},x})_{(f,g)}, \\
L''_{2}: 0 \to (O_{X_{2},x})_{(h,g)} \xrightarrow{(g,-(h+\varepsilon \dfrac{w_{1}}{f} + \varepsilon^{2} \dfrac{w_{2}}{f}))^{\mathrm{T}}} (O_{X_{2},x})_{(h,g)}^{\oplus 2} \xrightarrow{(h+\varepsilon \dfrac{w_{1}}{f} + \varepsilon^{2} \dfrac{w_{2}}{f}, g)} (O_{X_{2},x})_{(h,g)}.
\end{CD}
\end{cases}
\]

The image of $L''_{1}$ under the Ch map in Theorem \ref{theorem: firstorder}(take j=2)
\[
\mathrm{Ch}: \bigoplus\limits_{y \in X^{(2)}}K_{0}(O_{X_{2},y} \ \mathrm{on} \ y) \to \bigoplus\limits_{y \in X^{(2)}}H_{y}^{2}((\Omega_{X/\mathbb{Q}}^{1})^{\oplus 2})
\]
may be described similarly as the Ch map in Theorem ~\ref{theorem: TheoremKernel1} and is represented by the following diagram,
\[
\begin{cases}
 \begin{CD}
   F_{\bullet}(f,g)  @>>> (O_{X,x})_{(f,g)}/(f,g) @>>> 0  \\
   (O_{X,x})_{(f,g)} @> \frac{w_{1}}{h}dg + \frac{w_{2}}{h} dg >> (\Omega^{1}_{(O_{X,x})_{(f,g)}/\mathbb{Q}})^{\oplus 2},
 \end{CD}
\end{cases}
\]
where $F_{\bullet}(f,g) $ is the Koszul complex
\[
(O_{X,x})_{(f,g)} \xrightarrow{(g,-f)^{\mathrm{T}}}  (O_{X,x})_{(f,g)}^{\oplus 2} \xrightarrow{(f,g)} (O_{X,x})_{(f,g)}.
\]
By mimicking the argument in Theorem ~\ref{theorem: TheoremKernel1}, we can show

\begin{theorem} \label{theorem:theoremKernel2}
In the notation above, $L''$ is a Milnor K-theoretic cycle defined in Definition \ref{definition: Milnor K-theoretic Chow groups},
\[
L'' \in Z^{M}_{2}(D^{\mathrm{Perf}}(X_{2})).
\]
It is obvious that $L''$ is a deformation of  $L'(3.1) $ in the sense of Definition \ref{definition: deformation}.
\end{theorem}

In the notation of Set-up(page 4), for a first order infinitesimal deformation $Y'$ of $Y$, which is generically given by $(f+\varepsilon f_{1}, g)$,
let $\mu(Y') =F_{\bullet}(f+\varepsilon f_{1}, g) \in K_{0}(O_{X_{1},y} \ \mathrm{on} \ y)_{\mathbb{Q}}$ denote the image of $Y'$ under the map $\mu$ in Definition \ ~\ref{definition: map1}, which is the Koszul complex associated to $(f+ \varepsilon f_{1},g)$
\[
 0 \to O_{X_{1},y} \xrightarrow{(g,-(f+\varepsilon f_{1} ))^{\mathrm{T}}} O_{X_{1},y}^{\oplus 2} \xrightarrow{(f+ \varepsilon f_{1}, g)} O_{X_{1},y}.
\]

Since  $O_{X,y}= (O_{X,x})_{(f,g)}$, we write
$f_{1}=\dfrac{a_{1}}{b_{1}} \in O_{X,y}$, where $a_{1},b_{1} \in O_{X,x}$ and $b_{1} \notin  (f,g)$. Then $b_{1}$ is either in or not in the maximal idea $(f,g,h) \subset O_{X,x}$.

\begin{lemma}  \label{lemma: trivialdeform}
In the notation above, if $b_{1} \notin (f,g,h)$, $\mu(Y')$ is a Milnor K-theoretic 
cycle defined in Definition \ref{definition: Milnor K-theoretic Chow groups},
\[
\mu(Y') \in Z^{M}_{2}(D^{\mathrm{Perf}}(X_{1})).
\]
\end{lemma}

\begin{proof}
If $b_{1} \notin (f,g,h)$, then $b_{1}$ is a unit in $O_{X,x}$, this says $f_{1}= \dfrac{a_{1}}{b_{1}} \in O_{X,x}$.

As explained in Theorem \ref{theorem: TheoremKernel1}(page 10),
the image of $\mu(Y')$ in $H_{y}^{2}(\Omega^{1}_{X /\mathbb{Q}})$  under the Ch map   
\[
\mathrm{Ch}: \bigoplus\limits_{y \in X^{(2)}}K_{0}(O_{X_{1},y} \ \mathrm{on} \ y) \to \bigoplus\limits_{y \in X^{(2)}}H_{y}^{2}(\Omega_{X/\mathbb{Q}}^{1}),
\]
can be represented by the following diagram:
\[
\begin{cases}
 \begin{CD}
   (O_{X,x})_{(f,g)} @>(g,-f)^{\mathrm{T}}>>  (O_{X,x})_{(f,g)}^{\oplus 2} @>(f,g)>> (O_{X,x})_{(f,g)} @>>> (O_{X,x})_{(f,g)}/(f,g)@>>> 0  \\
   (O_{X,x})_{(f,g)} @>f_{1}dg >> \Omega^{1}_{(O_{X,x})_{(f,g)}/\mathbb{Q}}.
 \end{CD}
\end{cases}
\]

Noting $f_{1}dg = \frac{f_{1}h dg}{h}$, 
$\partial^{2,-2}_{1}$ in Theorem \ref{theorem: firstorder} maps
 $\beta$ in $H^{3}_{x}(\Omega^{1}_{X/\mathbb{Q}})$ to the following diagram:
\begin{equation}
\begin{cases}
 \begin{CD}
   O_{X,x} @>M_{1}>>  O_{X,x}^{\oplus 3} @>M_{2}>> O_{X,x}^{\oplus 3}  @>M_{3}>> O_{X,x} @>>> O_{X,x}/(f, g,h) @>>> 0  \\
   O_{X,x} @>f_{1}h dg >> \Omega^{1}_{O_{X,x}/\mathbb{Q}},
 \end{CD}
\end{cases}
\end{equation}
where $M_{1},  M_{2}$ and $M_{3}$ are matrices associated to the Koszul resolution of $O_{X,x}/(f, g,h)$:
\[
M_{1}=\begin{pmatrix} f \\ -g \\ h \end{pmatrix}, M_{2}=\begin{pmatrix} 0  & -h & -g \\ -h & 0 & f \\ g &f &0 \end{pmatrix}, M_{3}=(f,g,h).
\]

Since $h$ appears in $M_{1}$,  $\partial^{2,-2}_{1}(\beta)=0 $ in $H^{3}_{x}(\Omega^{1}_{X/\mathbb{Q}})$. 

Therefore, $d^{2,-2}_{1,X_{1}}(\mu(Y')) = 0$ because of the commutative diagram in Theorem \ref{theorem: firstorder}(take j=1).

\end{proof}

\begin{theorem} \label{theorem: mainTheorem}
In the notation of Set-up(page 4), for a first order infinitesimal deformation $Y'$ of $Y$ which is generically given by $(f + \varepsilon f_{1}, g)$, where $f_{1}= \dfrac{a_{1}}{b_{1}} \in O_{X,y}=(O_{X,x})_{(f,g)}$, 

\begin{itemize}
\item Case 1: if $b_{1} \notin (f,g,h)$, then $\mu(Y') \in  Z^{M}_{2}(D^{\mathrm{Perf}}(X_{1}))$ and $\mu(Y')$ deforms to  second order in $Z^{M}_{2}(D^{\mathrm{Perf}}(X_{2}))$ in the sense of Definition \ref{definition: deformation}.\\

\item Case 2: if $b_{1} \in (f,g,h)$,  we can find another curve $Z \subset X$ and a first order infinitesimal deformation $Z'$ of $Z$, such that 

1. $\mu(Y') + \mu(Z') \in Z^{M}_{2}(D^{\mathrm{Perf}}(X_{1}))$, 

2. the first order deformation  $\mu(Y') + \mu(Z') $ deforms to second order in $Z^{M}_{2}(D^{\mathrm{Perf}}(X_{2}))$.

\end{itemize}

\end{theorem}

\begin{proof}
Case 1 follows from Lemma ~\ref{lemma: trivialdeform}. We have proved that $\mu(Y') \in  Z^{M}_{2}(D^{\mathrm{Perf}}(X_{1}))$ in Lemma ~\ref{lemma: trivialdeform}. By mimicking the proof of Lemma ~\ref{lemma: trivialdeform}, we can show that the following complex
 \[
 0 \to O_{X_{2},y} \xrightarrow{(g,-(f+\varepsilon f_{1} + \varepsilon^{2} f_{2} ))^{\mathrm{T}}} O_{X_{2},y}^{\oplus 2} \xrightarrow{(f+ \varepsilon f_{1} + \varepsilon^{2} f_{2}, g)} O_{X_{2},y},
\]
where $f_{2} \in O_{X,x}$, is in $Z^{M}_{2}(D^{\mathrm{Perf}}(X_{2}))$ and it is a deformation of $\mu(Y')$, in the sense of Definition \ref{definition: deformation}.

Now, we consider the case $b_{1} \in (f,g,h)$. Since $b_{1} \notin (f, g)$,
we can write  $b_{1}= b_{f}f + b_{g}g+uh^{n}$, where $b_{f},b_{g} \in O_{X,x}$, $u$ is a unit in $O_{X,x}$ 
and $n$ is some positive integer. For simplicity, we assume $u=1$ and $n=1$.

The ideal $(f+\varepsilon f_{1}, g) \subset O_{X,y}[\varepsilon]$ can be expressed as 
\[
(f+\varepsilon \dfrac{a_{1}}{b_{f}f + b_{g}g+h}, g) = (f+\varepsilon \dfrac{a_{1}}{b_{f}f +h}, g) = (f+\varepsilon \dfrac{a_{1}}{h}, g).
\]
So we reduce to looking at $(f+\varepsilon \dfrac{a_{1}}{h}, g)$. Let $Z$ be the curve defined in Definition ~\ref{definition: definingCurveZ} and $Z'$ be a first order infinitesimal deformation of $Z$, which is generically given by $(h+\varepsilon \dfrac{a_{1}}{f}, g)$.

By taking $w_{1} = a_{1} \in O_{X,x}$, the Koszul complex $L'$(3.1) is of the form
\begin{equation}
 0 \to O_{X_{1},x} \xrightarrow{(g,-(fh+\varepsilon a_{1}))^{\mathrm{T}}} O_{X_{1},x}^{\oplus 2} \xrightarrow{(fh+ \varepsilon a_{1}, g)} O_{X_{1},x}.
\end{equation}
Under the isomorphism in Theorem \ref{theorem: Balmer theorem}(take j=1), the complex $L'$ decomposes into the direct sum of $\mu(Y')$  and $\mu(Z')$: 
\begin{equation}
L'  = \mu(Y') + \mu(Z'),
\end{equation}
where $\mu(Y')$  and $\mu(Z')$ are of the forms
\[
\begin{cases}
 \begin{CD}
\mu(Y'): 0 \to (O_{X_{1},x})_{(f,g)} \xrightarrow{(g,-(f+\varepsilon \dfrac{a_{1}}{h}))^{\mathrm{T}}} (O_{X_{1},x})_{(f,g)}^{\oplus 2} \xrightarrow{(f+\varepsilon \dfrac{a_{1}}{h}, g)} (O_{X_{1},x})_{(f,g)}, \\
\mu(Z'): 0 \to (O_{X_{1},x})_{(h,g)} \xrightarrow{(g,-(h+\varepsilon \dfrac{a_{1}}{f}))^{\mathrm{T}}} (O_{X_{1},x})_{(h,g)}^{\oplus 2} \xrightarrow{(h+\varepsilon \dfrac{a_{1}}{f}, g)} (O_{X_{1},x})_{(h,g)}.
\end{CD}
\end{cases}
\]
According to Theorem ~\ref{theorem: TheoremKernel1}, $\mu(Y') + \mu(Z') \in Z^{M}_{2}(D^{\mathrm{Perf}}(X_{1}))$.

By taking $w_{1} = a_{1} \in O_{X,x}$ and $w_{2} = a_{2} \in O_{X,x}$, the Koszul complex $L'' $(3.5) is of the form
\begin{equation}
 0 \to O_{X_{1},x} \xrightarrow{(g,-(fh+\varepsilon a_{1} +\varepsilon^2 a_{2} ))^{\mathrm{T}}} O_{X_{1},x}^{\oplus 2} \xrightarrow{(fh+ \varepsilon a_{1} +\varepsilon^2 a_{2}, g)} O_{X_{1},x}.
\end{equation}

According to Theorem ~\ref{theorem:theoremKernel2}, $L''(3.9) \in Z^{M}_{2}(D^{\mathrm{Perf}}(X_{2}))$ and  it is obvious that $L''(3.9) $ is a deformation of $L'(3.7)$.

\end{proof}

Let $L$ denote the following complex, which is the Koszul resolution of $O_{X,x}/(fh,g)$ \footnote{We naively think this defines a nodal curve.},
\begin{equation} 
 0 \to O_{X,x} \xrightarrow{(g,-fh)^{\mathrm{T}}} O_{X,x}^{\oplus 2} \xrightarrow{(fh, g)} O_{X,x}.
\end{equation}
Under the isomorphism in Theorem \ref{theorem: Balmer theorem}(take j=0), $L$ decomposes into the direct sum of $\mu(Y)$ and $\mu(Z)$, 
\[
\begin{cases}
\begin{CD}
\mu(Y): 0 \to (O_{X,x})_{(f,g)}\xrightarrow{(g,-f)^{\mathrm{T}}} (O_{X,x})_{(f,g)}^{\oplus 2} \xrightarrow{(f, g)} (O_{X,x})_{(f,g)}, \\
\mu(Z): 0 \to (O_{X,x})_{(g,h)} \xrightarrow{(g,-h)^{\mathrm{T}}} (O_{X,x})_{(g,h)}^{\oplus 2} \xrightarrow{(h, g)} (O_{X,x})_{(g,h)}.
\end{CD}
\end{cases}
\]

\begin{theorem} \label{theorem: answerToNg}
The answer to Question \ ~\ref{question: NgRewrite} is positive. To be precise, in the same assumption of Theorem \ref{theorem: mainTheorem},
\begin{itemize}
\item Case 1: if $b_{1} \notin (f,g,h)$, then $\mu(Y') \in  Z^{M}_{2}(D^{\mathrm{Perf}}(X_{1}))$ and $\mu(Y')$ deforms to  second order in $Z^{M}_{2}(D^{\mathrm{Perf}}(X_{2}))$.\\

\item Case 2: if $b_{1} \in (f,g,h)$\footnote{In this case, in general, $\mu(Y')$ is not a Milnor K-theoretic so that it is not a deformation of $\mu(Y)$.},  
we can find an element $L \mathrm{(3.10)} \in Z^{M}_{2}(D^{\mathrm{Perf}}(X))$ satisfying $L= \mu(Y) + \mu(Z)$, where the curve $Z \subset X$ defined in Definition \ref{definition: definingCurveZ},  a first order  deformation $L' \mathrm{(3.7)}$ of $L\mathrm{(3.10)}$, in the sense of  Definition \ref{definition: deformation} such that 

(1). $ L' \mathrm{(3.7)} = \mu(Y') + \mu(Z') \in Z^{M}_{2}(D^{\mathrm{Perf}}(X_{1}))$, with $Z'$ a first order infinitesimal deformation of $Z$. So it is obvious that $L' \mathrm{(3.7)} \mid_{Y} = \mu(Y')$. 

(2). the first order deformation $L' \mathrm{(3.7)}$ extends to second order $L'' \mathrm{(3.9)} \in Z^{M}_{2}(D^{\mathrm{Perf}}(X_{2}))$.

(3). Considering $\mu(Z)$ as a first order deformation of itself, we can 
extend it to second order $\mu(Z)$(meaning we fix $Z$).

In conclusion, $(\mu(Y') + \mu(Z'))-\mu(Z)$ is a first order deformation of $\mu(Y)$ satisfying $((\mu(Y') + \mu(Z'))-\mu(Z))\mid_{Y}=\mu(Y')$ 
and we can deform it to second order.

\end{itemize}
\end{theorem}

\section{Acknowledgements}
\label{Acknowledge}
The author must record that the main ideas in this note are known to Mark Green and Phillip Griffiths \cite{GGtangentspace} and TingFai Ng \cite{Ng}(for the divisor case). This note is to reinterpret and generalize their ideas in a different language.

The author thanks Spencer Bloch for comments and questions on previous version and thanks
Christophe Soul\'e for comments on \cite{Y-3, Y-4}; thanks Shiu-Yuen Cheng and Congling Qiu for very helpful correspondence.


\end{document}